\documentclass[12pt]{amsart}

\usepackage{amssymb}
\usepackage{amscd}
\usepackage[bookmarks=true]{hyperref}
\hypersetup{colorlinks,citecolor=blue,linkcolor=blue}

\newtheorem{theorem}{Theorem}[section]

\newtheorem{prop}[theorem]{Proposition}
\newtheorem{cor}[theorem]{Corollary}
\theoremstyle{definition}

\newtheorem{rmk}[theorem]{Remark}

\newcommand\Hy{\mathbb{H}}

\newcommand\pr{\mathrm{P}\rho}

\newcommand\Z{\mathbb{Z}}
\newcommand\Q{\mathbb{Q}}
\newcommand\R{\mathbb{R}}
\newcommand\C{\mathbb{C}}
\newcommand\cO{\mathcal{O}_k}
\newcommand\cOl{\mathcal{O}_\ell}
\newcommand\cQ{\mathcal{Q}}

\newcommand\area{{\rm area}}
\newcommand\sys{{\rm sys}}
\newcommand\sysg{{\rm sysg}}
\newcommand\hg{{\rm Hg}}
\newcommand\rk{{\rm rk}}
\newcommand\vol{{\rm vol}}
\newcommand\Nfr{\mathcal{N}_{fr}}

\newcommand{\pinj}{$\pi_1$-injective }

\DeclareMathOperator{\SO}{SO}
\DeclareMathOperator{\PO}{PO}

\DeclareMathOperator{\PSL}{PSL}

\providecommand{\norm}[1]{\lVert#1\rVert}

\begin{document}
\title{On $2$-systoles of hyperbolic $3$-manifolds}


\author{Mikhail Belolipetsky}\thanks{Supported by a CNPq research grant.}
\address{
IMPA\\
Estrada Dona Castorina, 110\\
22460-320 Rio de Janeiro, Brazil}
\email{mbel@impa.br}

\begin{abstract}
We investigate the geometry of \pinj surfaces in closed hyperbolic $3$-manifolds. First we prove that for any $\epsilon>0$, if the manifold $M$ has sufficiently large systole $\sys_1(M)$, the genus of any such surface in $M$ is bounded below by $\exp((\frac12-\epsilon)\sys_1(M))$. Using this result we show, in particular, that for congruence covers $M_i\to M$ of a compact arithmetic hyperbolic $3$-manifold we have: (a) the minimal genus of \pinj surfaces satisfies $\log \sysg(M_i) \gtrsim \frac13 \log \vol(M_i)$; (b) there exist such sequences with the ratio Heegard genus$(M_i)/\sysg(M_i) \gtrsim \vol(M_i)^{1/2}$; and  (c) under some additional assumptions $\pi_1(M_i)$ is $k$-free with $\log k \gtrsim \frac13\sys_1(M_i)$. The latter resolves a special case of a conjecture of M.~Gromov.
\end{abstract}

\maketitle

\section{Introduction}

Let $M$ be a closed hyperbolic $3$-manifold defined arithmetically and $M_i \to M$ be a sequence of its congruence covers. We will recall the precise definitions later on in the paper restricting our attention by now to some general principles. The arithmetically defined manifolds $M_i$ possess a set of remarkable geometrical and topological properties. Their systole length grows proportionally to the logarithm of the volume, which is asymptotically the fastest possible growth. An analogues property is known to hold for the girth of certain families of expander graphs. In \cite{Gromov:GAFA}, Gromov proved that any generic smooth map $F: M_i \to \R$ has a fiber $F^{-1}(x)\subset M_i$ with the sum of its Betti numbers $\gtrsim \vol(M_i)$. He compares this complicated fiber property with the basic property of expanders: if $\Gamma_i$ is a family of expander graphs and $\Gamma_i$ has $N_i$ edges, then any continuous map $F: \Gamma_i \to \R$ has a fiber that meets $\gtrsim N_i$ different edges. The conclusion is that the arithmetically defined families of hyperbolic $3$-manifolds can be considered as topological analogues of expanders.

Congruence subgroups of arithmetic lattices in $\PSL_2(\C)$ and associated hyperbolic $3$-manifolds were actively studies by many people from various viewpoints. Even a brief account of related work would take us far beyond the scope of this introduction. Regardless of the large amount of research in this area and important breakthroughs, many fundamental questions about the families of arithmetically defined hyperbolic $3$-manifolds remain open. In particular, we know surprisingly little about the structure of their fundamental groups. The purpose of the present paper is to obtain some new results in this direction.

The main result of the paper is Theorem~\ref{thm1}, which applies to any closed hyperbolic $3$-manifold $M$ regardless of arithmeticity. The theorem says that for any $\epsilon>0$, if the systole $\sys_1(M)$ ($=$ the length of a shortest closed geodesic in $M$) is sufficiently large, then
$$\sysg(M) \ge \exp((\frac12-\epsilon)\sys_1(M)),$$
where $\sysg(M)$ denotes the minimal genus of a surface subgroup of $\pi_1(M)$ which we call the \emph{systolic genus of} $M$. The notion of systolic genus is related to the $2$-dimensional systole $\sys_2(M)$, which can be defined as the minimal area of a \pinj surface of genus $g > 0$ in the manifold $M$ (in terminology of \cite{Gromov:lectures} it is called the \emph{absolute $2$-systole}). In fact, one of the key ingredients of the proof of Theorem~\ref{thm1} is Thurston's inequality bounding $\sys_2(M)$ through the genus. Another main ingredient of the proof is Gromov's systolic inequality for surfaces of high genus, in which we use a numerical value of the constant provided by the work of Katz--Sabourau \cite{KS}. The application of these ingredients is supported by fundamental results from the theory of minimal surfaces of Schoen--Yau \cite{SY} and Sacks--Uhlenbeck \cite{SU}. After proving Theorem~\ref{thm1} and discussing some of its immediate corollaries in Section~\ref{sec:genus}, we proceed to its applications to the congruence covers.

Let us make a small pause here and recall a recent preprint by DeBlois \cite{DB}, which was one of the main motivations for the present article. Generalizing a result of Lackenby \cite{L}, DeBlois showed that the rank of the fundamental group of cyclic covers grows linearly with the degree if and only if the the covers satisfy a certain homological condition. This result supports the conjecture saying that given a sequence of finite covers $M_i\to M$ of a $3$-manifold $M$, the types of growth of the Heegard genus $\hg(M_i)$ and the rank of the fundamental group $\rk(M_i)$ should be the same. For the congruence covers it was proved independently by Lackenby \cite{L} and Gromov \cite{Gromov:GAFA} that the Heegard genus grows linearly with the degree, hence we may expect the linear growth of the minimal number of generators of the fundamental groups of the covers. In his proof DeBlois uses the fact that under certain conditions the covering manifolds in the sequence of cyclic covers all contain an incompressible surface of a fixed genus. These surfaces are pivotal to the argument in \cite{DB}, and the natural question arising here is that perhaps one can detect similar surfaces in the congruence covers. The results of Section~\ref{sec:congruence} show that this is impossible: the minimal genus of \pinj surfaces in the congruence covers $M_i \to M$ grows at least as $\vol(M_i)^{1/3}$ (see Proposition~\ref{prop1}). This result is proved using Theorem~\ref{thm1} and the lower bound for the growth of systole along the congruence covers of an arithmetic $3$-manifold from \cite{KSV}. The constant $\frac13$ in the proposition (as well as $\frac12$ in the theorem) is probably not optimal. The second proposition in Section~\ref{sec:congruence} shows that in general this constant cannot be improved above $\frac12$, and hence the optimal constant in Theorem~\ref{thm1} lies in $[\frac12, \frac34]$. This proposition is proved using totally geodesic immersions and the strong approximation theorem.

The conjecture about the growth rate of the rank of the fundamental group of congruence covers is still open, but in Section~\ref{sec:free} we will prove a somewhat stronger property with a weaker exponent. Recall that a group $\Gamma$ is called \emph{$k$-free} if any subgroup of $\Gamma$ generated by $k$ elements is free. We denote the maximal $k$ for which $\Gamma$ is $k$-free by $\Nfr(\Gamma)$. In \cite[Section~5.3.A]{Gromov:hyp groups}, Gromov stated that $\Nfr(\Gamma)$ is bounded below by an exponential function of the systole (or injectivity radius) of the associated quotient space $M$. The details of the proof were not given in \cite{Gromov:hyp groups}, they can be found in a later paper \cite[Section~2.4]{Gromov:GAFA} where it is pointed out that the argument gives only a bound of the form $\epsilon r/\log(r)$, $r=\sys_1(M)$. Two other proofs of the growth of $\Nfr(\Gamma)$ when $\sys_1(M) \to \infty$ appear in \cite{Arzh} and \cite{KW}, but the quantitative bounds for $\Nfr(\Gamma)$ which can be deduced from these papers are weaker than the one above. Although the difference between sub-linear and exponential growth is very large, in \cite[p.~763]{Gromov:GAFA} Gromov conjectured that the true bound should be exponential. He pointed out that this is not known even for the fundamental groups of hyperbolic $3$-manifolds. Using the methods of this paper in Section~\ref{sec:free} we show that under some mild arithmetic assumptions, for a sequence of congruence covers $M_i \to M$ of a compact arithmetic hyperbolic $3$-manifold $M$ we have
$$\log \Nfr(\pi_1(M_i)) \gtrsim \frac13\sys_1(M_i),\text{ as } i\to\infty.$$
The precise statements of the results are given in Theorem~\ref{thm2} and Corollary~\ref{cor:gromov}. This resolves a special case of the conjecture about the growth of $\Nfr$. It also provides the best currently known lower bound for $\rk(M_i)$.

In Section~\ref{sec:questions} we discuss higher dimensional generalizations of the results and some other related questions.

\medskip

\noindent\textsc{Notation.} Throughout the paper, the relation $f(x) \gtrsim g(x)$ for two positive functions $f(x)$ and $g(x)$ means that for any $\epsilon > 0$ there exists $x_0 = x_0(\epsilon)$ such that for all $x\ge x_0$ we have $f(x) \ge (1-\epsilon)g(x)$. We write $f(x) \sim g(x)$ if $f(x) \gtrsim g(x)$ and $g(x) \gtrsim f(x)$.

\medskip

\noindent\textbf{Acknowledgments.} I want to thank Nicolas Bergeron and Mikhail Gromov for correspondence, and Harold Rosenberg for helpful discussions. I also want to thank the referee for valuable comments and corrections.

\section{Systolic genus} \label{sec:genus}

Let $M$ be an $n$-dimensional Riemannian manifold and $S_g$ denote a Riemann surface of positive genus. We define the \emph{systolic genus} of $M$ by
$$\sysg(M) = \min\{g \mid \text{the fundamental group } \pi_1(M) \text{ contains } \pi_1(S_g)\},$$
and the \emph{homological systolic genus}
$$\sysg^h(M) = \min\{g \mid \text{there exists } \sigma\in H_2(M, \Z)\setminus\{0\} \text{ represented by } S_g\}.$$
Here we assume by convention that $\min\{\emptyset\} = \infty$.

It is clear from the definitions that $\sysg(M) \le \sysg^h(M)$. By the work of Kahn--Markovic \cite{KahM}, we know that any hyperbolic $3$-manifold contains a $\pi_1$-injective surface, therefore we have $\sysg(M^3) < \infty$ for any such manifold $M^3$. On the other hand we may have $H_2(M^3, \Z) = 0$, which would imply $\sysg^h(M^3) = \infty$. We can interpret $\sysg^h(M^3)$ as the minimal value of the Thurston norm \cite{T2} on $H_2(M^3, \Z)$ (see Corollary~\ref{cor:ThNorm} below).

The notion of the systolic genus is related to the $2$-dimensional systole $\sys_2(M)$, which can be defined as the minimal area of a \pinj surface in $M$ of genus $g > 0$ with respect to the induced metric. We refer to \cite{Gromov:lectures} or \cite{Katz:book} for more information about systoles. One of the corollaries of the considerations below is that for closed hyperbolic $3$-manifolds there is an equivalence between the systolic genus and $2$-dimensional systole (see Section~\ref{sec:questions2} for the details and references to the related work).

\begin{theorem}\label{thm1}
Let $M$ be a closed hyperbolic $3$-manifold. For any $\epsilon > 0$, assuming that the systole $\sys_1(M)$ is sufficiently large, we have
$$\sysg(M) \ge e^{(\frac12-\epsilon)\sys_1(M)}.$$
In particular, given a sequence of closed hyperbolic $3$-manifolds $M$ with $\sys_1(M) \to \infty$, we have
$$\log \sysg(M) \gtrsim \frac12\sys_1(M).$$
\end{theorem}

\begin{proof} Assume $\sysg(M) = g$ and so $\pi_1(M)$ has a subgroup isomorphic to $\pi_1(S_g)$, $g>0$. Then there exists a continuous map $f : S_g\to M$, which is $\pi_1$-injective and whose image is a (singular) surface of genus $g$ immersed into $M$. This map can be obtained by a general constructibility argument as e.g.~in \cite[Proof of Theorem~5.2]{SY}. Now by a theorem of Schoen--Yau \cite{SY} or Sacks--Uhlenbeck \cite{SU}, there exists a minimal immersion $h: S_g\to M$ such that $h_* = f_*$ on $\pi_1(S_g)$ and the induced area of $h$ is the least among all maps with the same action on $\pi_1(S_g)$. Note that the conclusion about $h$ being an immersion without branched points requires the assumption that the manifold $M$ is $3$-dimensional.

The pullback by $h$ defines a Riemannian metric $\rho$ on $S_g$ with
\begin{equation}\label{eq0}
\sys_1(S_g, \rho) \ge \sys_1(M).
\end{equation}

We now recall an important high genus systolic inequality due to Gromov \cite{Gromov:filling} (see also \cite[p.~88]{Katz:book}):
\begin{equation}\label{eq1}
\frac{\sys_1(S_g, \rho)}{\log g} \lesssim C\left(\frac{\area(S_g, \rho)}{g}\right)^\frac12, \text{ as } g\to\infty.
\end{equation}
The best known asymptotic value of the constant $C = \frac{1}{\sqrt{\pi}}$ was obtained by Katz--Sabourau in \cite{KS}.

On the face of it, inequality \eqref{eq1} provides precisely what we need for the proof of the theorem but there is one delicate issue here: the inequality holds only for surfaces of large genus  which is not part of our assumptions. Indeed, it is rather our goal to show that the genus $g$ has to be large together with a quantitative version of this statement. We will return to inequality \eqref{eq1} later after resolving this issue.

The area of an immersed \pinj minimal surface $S_g \subset M$ of genus $g > 1$ satisfies Thurston's inequality
\begin{equation}\label{eq2}
\area(S_g,\rho) \le 4\pi(g-1).
\end{equation}
This inequality is obtained by applying the Gauss equation to the minimal surfaces in hyperbolic $3$-manifolds and then using the Gauss--Bonnet theorem (see e.g. \cite[Lemma~6]{Hass}). The simplicity of the argument in no way does reduce the importance of this remarkable inequality. We will discuss its generalizations and some related problems in Section~\ref{sec:questions}.

We now can show that $\sys_1(M)\to\infty$ implies $\sysg(M)\to\infty$ which would enable us to use \eqref{eq1}. In order to do so we recall a well-known Besicovitch's systolic inequality:
\begin{equation}\label{eq3}
\sys_1(S_g, \rho)^2 \le 2\,\area(S_g, \rho).
\end{equation}
The proof of this inequality is given in \cite[Section~4.5]{Gromov:Metric Structures}. Inequality \eqref{eq3} is asymptotically weaker than \eqref{eq1} but the advantage is that it holds for any $g \ge 1$. By the assumption, $M$ is a closed hyperbolic $3$-manifold therefore it does not contain \pinj tori and $g=1$ does not arise. Combining \eqref{eq3} with \eqref{eq2}, we obtain
$$\sys_1(S_g, \rho)^2 \le 2\,\area(S_g, \rho) \le 8\pi(g-1).$$
Hence by \eqref{eq0}, $\sys_1(M)\to\infty$ implies $g \to \infty$.

To finish the proof we can use inequality \eqref{eq1} together with \eqref{eq2}. By \eqref{eq1} and the previous argument, given $\epsilon_0 > 0$, there exists $s = s(\epsilon_0)$ such that if $\sys_1(M) > s$, then
\begin{equation*}
\frac{\sys_1(S_g, \rho)}{\log g} \le \left(\frac{\area(S_g, \rho)}{(\pi-\epsilon_0)g}\right)^\frac12.
\end{equation*}
Hence we have
\begin{align*}
\log g  &\ge \sys_1(S_g, \rho) \left(\frac{(\pi-\epsilon_0)g}{\area(S_g, \rho)}\right)^\frac12 \\
        &\ge \sys_1(S_g, \rho) \left(\frac{(\pi-\epsilon_0) g}{4\pi(g-1)}\right)^\frac12 \\
        &\ge \left(\frac{1}{2}-\epsilon\right) \sys_1(S_g, \rho) .
\end{align*}
By inequality \eqref{eq0}, this gives
\begin{equation*}
g \ge e^{(\frac12-\epsilon)\sys_1(M)}.
\end{equation*}
The asymptotic inequality for $\log\sysg(M)$ follows immediately.
\end{proof}

We remark that Theorem~\ref{thm1} also provides a lower bound for the homological systolic genus $\sysg^h(M)$, which can be viewed as a lower bound for the Thurston norm \cite{T2}. Let us state the result for the Thurston norm $\norm{.}_T$  as a corollary.
\begin{cor}\label{cor:ThNorm}
Under the assumptions of Theorem~\ref{thm1}, for every $\sigma\in H_2(M, \Z)\setminus\{0\}$ we have
$$\norm{\sigma}_T \ge 2e^{(\frac12-\epsilon)\sys_1(M)} - 2.$$
\end{cor}
\begin{proof}
It is easy to see that any non-trivial class $\sigma\in H_2(M, \Z)$ can be represented by a smoothly embedded surface $S$. The surface $S$ may be disconnected in which case we denote its connected components by $S_i$. The Thurston norm of $\sigma$ is defined by
$$\norm{\sigma}_T = \min\{\chi_-(S) \mid [S] = \sigma\}, \text{ where } \chi_-(S) = \sum_{g(S_i)>0} (2g(S_i)-2).$$
The result that for a \emph{closed} $3$-manifold $M$, $\norm{.}_T$ extends to a norm on $H_2(M, \R)$ is based on the fact that at least one of the surfaces $S_i$ in the decomposition above has $g>1$. It follows that $\pi_1(M)$ contains a subgroup isomorphic to $\pi_1(S_g)$ and so we can apply the theorem.
\end{proof}

A different class of lower bounds for the Thurston norm can be obtained using the Seiberg--Witten theory. We refer to \cite{Kr} for a nice survey of related results and methods. The bounds which come out this way are in a sense more precise, as, for instance, they usually depend (as they should) on the homology class $\sigma$. The main feature of our estimate is that it allows us to capture some global geometric information about the manifold $M$ in a single inequality that applies simultaneously to all the non-trivial classes.

\begin{cor}\label{cor:sys2}
Under the assumptions of Theorem~\ref{thm1}, we have
\begin{align*}
\log \sys_2(M)   &\gtrsim \frac12\sys_1(M); \text{ and}\\
\log \sys^h_2(M) &\gtrsim \frac12\sys_1(M), \text{ as \ }\sys_1(M)\to\infty.
\end{align*}
\end{cor}

\begin{proof} By \eqref{eq1}, we have
\begin{equation*}\
\area(S_g,\rho) \gtrsim \frac{g}{(C\log g)^2}\,\sys_1(S_g, \rho)^2, \text{ as \ }g\to\infty;
\end{equation*}
which implies
\begin{equation*}
\log \area(S_g,\rho) \gtrsim \log g.
\end{equation*}
The corollary now follows from the theorem and the fact that $\sysg^h(M) \ge \sysg(M)$.
\end{proof}

\section{Congruence covers} \label{sec:congruence}

The main application of Theorem \ref{thm1} is to the sequences of congruence covers of a compact arithmetic hyperbolic $3$-manifold. We first recall some definitions (see \cite{MR} for further details).

By an \emph{arithmetic hyperbolic $3$-manifold} we mean a quotient space $\Hy^3/\Gamma$, where $\Hy^3$ is the hyperbolic $3$-space and $\Gamma$ is a torsion-free arithmetic Kleinian group. The arithmetic Kleinian groups are defined as follows. Let $k$ be a number field having exactly one complex place, $\cO$ the ring of integers of $k$, and $D$ a quaternion algebra over $k$ which ramifies at all real places of $k$. Let $\rho: D \to \mathrm{M}(2,\C)$ be an embedding, $\cQ$ an $\cO$-order of $D$, and $\cQ^1$ the elements of norm one in $\cQ$. Then $\pr(\cQ^1) < \PSL(2,\C)$ is a finite covolume Kleinian group, which is cocompact if and only if $D$ is not isomorphic to $\mathrm{M}(2, \Q[\sqrt{-d}])$, where $d$ is a square free positive integer. An \emph{arithmetic Kleinian group} $\Gamma$ is a subgroup of $\PSL(2,\C)$ commensurable with a group of the type $\pr(\cQ^1)$.

We now define congruence subgroups of an arithmetic Kleinin group $\Gamma$ which give rise to the congruence covers of the manifold $M = \Hy^3/\Gamma$. We can assume without loss of generality that $\Gamma = \pr(\cQ^1)$ for some $\cQ$ as above. Let $I$ be an ideal of the ring of integers $\cO$. It defines an ideal $I\cQ$ of the $\cO$-order $\cQ$, which yields a finite quotient ring $\cQ/I\cQ$. The principal congruence subgroup of $\cQ^1$ with respect to an ideal $I$ in $\cO$ is by definition the kernel of the homomorphism $\cQ^1 \to (\cQ/I\cQ)^\times$ induced by the natural projection $\cQ \to \cQ/I\cQ$. We denote it by $\cQ^1(I)$. The projection $\Gamma(I) = \pr(\cQ^1(I))$ is called a \emph{(principal) congruence subgroup} of  $\PSL(2,\C)$. It is a finite index subgroup of the group $\Gamma$, to which we associate a \emph{congruence cover} $M_I = \Hy^3/\Gamma(I) \to M$.

\begin{prop}\label{prop1}
A sequence of congruence covers  $M_I \to M$  of a compact arithmetic hyperbolic $3$-manifold $M$ satisfies
$$\log \sysg(M_I) \gtrsim \frac13 \log \vol(M_I), \text{ as } \mathrm{Norm}(I) \to \infty.$$
\end{prop}

\begin{proof}
By \cite[Theorem 1.8]{KSV}, we have
$$
\sys_1(M_I) \ge \frac23\log\vol(M_I) - c,\quad c = c(M).
$$
(In \cite{KSV} the inequality is stated with the simplicial volume $\norm{M_I}$ instead of the hyperbolic volume of $M_I$, which is equivalent since for a closed hyperbolic $3$-manifold $M$ we have $\norm{M} = \frac{\vol(M)}{v_3}$, where $v_3$ is the volume of a regular ideal simplex in $\Hy^3$.)

Hence
$$
\sys_1(M_I) \gtrsim \frac23\log\vol(M_I), \text{ when } \mathrm{Norm}(I) \to \infty.
$$
Using this inequality in Theorem \ref{thm1} gives
$$
\log \sysg(M_I) \gtrsim \frac12\sys_1(M_I) \gtrsim \frac12\cdot\frac23\,\log\vol(M_I).
$$
\end{proof}

Given a lower bound for the growth of $\sysg(M_I)$ it is natural to ask how far is it from the actual growth of the systolic genus. In this context it is worthwhile to recall the result of Lackenby \cite{L} and Gromov \cite{Gromov:GAFA} saying that the Heegard genus $\hg(M_I)$ grows proportionally to $\vol(M_I)$. There is no direct connection between the systolic genus and the Heegard genus of $M$, but still we may try to seek for a linear or even stronger lower bound for the former. The result of the next proposition shows that in many cases the growth rate of $\sysg(M_I)$ is actually sub-linear and, in particular, much slower than the growth rate of the Heegard genus.

\begin{prop}\label{prop2}
Let $M$ be a compact arithmetic hyperbolic $3$-manifold which contains an immersed totally geodesic hypersurface. Then $M$ has a sequence of congruence covers $M_I\to M$ such that
$$\sysg(M_I) \lesssim \vol(M_I)^{1/2}, \text{ as } \mathrm{Norm}(I) \to \infty.$$
\end{prop}
\begin{proof}
The assumption that $M$ has an immersed totally geodesic surface implies that $M$ belongs to the special class of arithmetic $3$-manifolds which are defined by quadratic forms (see \cite[Theorem~10.2.3]{MR}). Without loss of generality, we can assume that the defining form $f$ is a diagonal quadratic form $-a_0x_0^2 + a_1x_1^2 + a_2x_2^2 + a_3x_3^2$, with $a_i> 0$ algebraic integers in a totally real field $\ell\subset\R$ such that for the Galois embeddings $\sigma\colon \ell\hookrightarrow\R$ which are different from the original embedding, $\sigma(a_0) < 0$ and $\sigma(a_i) > 0$, $i = 1,2,3$, and that $\Gamma = \pi_1(M)$ is the group of $\cOl$-units $\SO^\circ(f,\cOl)$. These assumptions may require passing to a finite sheet cover of $M$, which is consistent with the proposition. Let $f' = -a_0x_0^2 + a_1x_1^2 + a_2x_2^2$ be a $3$-dimensional quadratic form obtained by the restriction of $f$ to the first $3$ coordinates, and let $\Gamma'$ be the group of $\cOl$-units of $f'$. The natural inclusion $\SO^\circ(f') \hookrightarrow \SO^\circ(f)$ is defined over $\ell$, and we have $\Gamma' = \Gamma\cap\SO^\circ(f')$.

A well known corollary of the strong approximation theorem (cf.~\cite[Theorem~7.7.5]{MR}) implies that for almost all ideals $I \subset \cOl$, the reduction maps $\Gamma\to\SO^\circ(f,\cOl/I)$ and $\Gamma'\to\SO^\circ(f',\cOl/I)$ are surjective, which leads to the following commutative diagram:
$$
\begin{CD}
1 @>>> \Gamma(I) @>>> \Gamma @>>> \SO^\circ(f,\cOl/I) @>>> 1 \\
@. @AAA @AAA @AAA @.\\
1 @>>> \Gamma'(I) @>>> \Gamma' @>>> \SO^\circ(f',\cOl/I) @>>> 1 \\
\end{CD}
$$
Hence we have
\begin{align*}
 [\Gamma:\Gamma(I)] &= |\SO^\circ(f,\cOl/I)| \sim q^6;\\
 [\Gamma':\Gamma'(I)] &= |\SO^\circ(f',\cOl/I)| \sim q^3,
\end{align*}
where $q$ denotes the order of the quotient ring $\cOl/I$. It follows that the $\area(M'(I))$, where $M'(I)$ is the totally geodesic surface in $M(I)$ associated to $\Gamma'(I)$, grows as $\vol(M(I))^{1/2}$, and hence its genus is also $\sim \vol(M(I))^{1/2}$.
\end{proof}

\begin{cor}
For the sequences of congruence covers from Proposition~\ref{prop2} we have
 $$\frac{\hg(M_I)}{\sysg(M_I)} \gtrsim \vol(M_I)^{1/2}, \text{ as } \mathrm{Norm}(I) \to \infty.$$
\end{cor}

As a consequence of Proposition \ref{prop2}, we can make some conclusions about the \emph{optimal constants} in Proposition~\ref{prop1} and Theorem~\ref{thm1}. Indeed, the proposition implies that in general the constant in Proposition~\ref{prop1} can not be improved above $\frac12$, and hence the optimal exponent in Theorem~\ref{thm1} lies in $[\frac12, \frac34]$. The quantitative part of the proof of Theorem~\ref{thm1} is based on inequalities \eqref{eq1} and \eqref{eq2}. It is known that inequality \eqref{eq2} is sharp with the equality attained for the totally geodesic surfaces. Regarding \eqref{eq1}, the best known examples are provided by the congruence covers first constructed by Buser and Sarnak in \cite{BSa}. They give the value $C = \frac{2}{3\sqrt{\pi}}$. If it would turn out that this is the true asymptotic value for the constant, then it would push the exponent in Theorem~\ref{thm1} up to $\frac34$, which is the best possible by Proposition~\ref{prop2}.

\begin{rmk}\label{rem}
Similarly to Corollary \ref{cor:sys2}, we can apply inequalities \eqref{eq1} and \eqref{eq2} to restate the results of this section in terms of $\sys_2(M)$.
\end{rmk}

\section{Free subgroups} \label{sec:free}

Recall that a group $\Gamma$ is called \emph{$k$-free} if every subgroup of $\Gamma$ which can be generated by $k$ elements is a free group (of some rank $\le k$). We will denote the maximal $k$ for which a finitely generated group $\Gamma$ is $k$-free by $\Nfr(\Gamma)$ or $\Nfr(M)$ if $\pi_1(M) = \Gamma$. The fundamental group of a hyperbolic manifold $M$ is torsion-free, hence for such $M$ we have $\Nfr(M) \ge 1$.

The following theorem was proved by Baumslag and Shalen:

\begin{theorem}\cite{BSh} \label{thm_shalen}  
Let $M$ be an irreducible, closed orientable $3$-manifold, and let $k$ be a positive integer. Suppose that $\pi_1(M)$ has no subgroup isomorphic to
$\pi_1(S_g)$ for any $g$ with $0 < g < k$, and that $\beta_1(M) > k$. Then $\pi_1(M)$ is $k$-free.
\end{theorem}
In a subsequent paper, Shalen and Wagreich obtained a similar result under an assumption on the rank of $H_1(M, \Z_p)$ instead of $H_1(M,\Z)$. The latter also can be used for the applications similar to the ones considered below, but we will not need it in this paper.

Assuming that we have the required information about the first Betti number, we can apply this theorem together with the results of the previous sections in order to give lower bounds for $\Nfr(\Gamma)$. The growth of the first Betti number of arithmetic quotient spaces was intensively studied during the last decades with several strong results appearing very recently. Important contributions were made by M.~Gromov, F.~Grunewald, J.~Millson, J.~Rohlfs, J.~Schwermer, and more recently by N.~Bergeron, M.~Cossutta, S.~Marshall, to name a few. Let us cite some corollaries of the results which are suitable for our applications:
\begin{quote}{\em
\begin{itemize}
\item[(1)] (Xue \cite{Xue}) If $M$ is an arithmetic hyperbolic $3$-manifold whose group is defined by a quadratic form and $M_i \to M$ is a sequence of its congruence covers, then $\log \beta_1(M_i) \gtrsim \frac13 \log\vol(M_i)$.
\item[(2)] (Calegari--Emerton \cite{CE}) If $M$ is a compact arithmetic hyperbolic $3$-manifold then for almost all primes $p$ the growth of the first Betti numbers in a sequence $M_i \to M$ of the congruence covers corresponding to a $p$-adic analytic tower satisfies $\log \beta_1(M_i) \gtrsim \frac56 \log\vol(M_i)$.
\item[(3)] (Kionke--Schwermer \cite{KiS}) If $M$ is an arithmetic $3$-manifold such that the underlying quaternion algebra meets certain conditions, then there exists a sequence of congruence covers $M_i \to M$ such that the first Betti numbers of $M_i$ satisfy $\log \beta_1(M_i) \gtrsim \frac12 \log\vol(M_i)$.
\end{itemize}
}\end{quote}
In all these cases we see that under some additional arithmetic assumptions the rank of  $H_1(M_i; \Z)$ grows at least as fast as $\vol(M_i)^\alpha$ with an explicit constant $1/3 \le \alpha < 1$. It is conjectured that similar asymptotic behavior takes place for any infinite sequence of congruence subgroups $\Gamma_i < \Gamma$ with $\cap_i\Gamma_i = \{\mathrm{id}\}$ (or, equivalently, with $\sys_1(\Hy^3/\Gamma_i) \to \infty$).

Together with Proposition \ref{prop1} and Theorem \ref{thm_shalen}, these estimates allow us to obtain the lower bounds for $\Nfr(M_i)$. The following theorem summarizes the results obtained this way:
\begin{theorem}\label{thm2}
Let $M_i \to M$ be a sequence of the congruence covers of a compact orientable arithmetic hyperbolic $3$-manifold $M$ which satisfies the assumptions in (1), (2), or (3). Then
$$\log \Nfr(M_i) \gtrsim \frac13 \log \vol(M_i), \text{ as } i\to\infty.$$
\end{theorem}

We now come back to a conjecture of Gromov discussed in the introduction. Recall that it says that $\Nfr(M)$ is expected to grow exponentially with respect to the systole $\sys_1(M)$. The systole of a hyperbolic $3$-manifold is bounded above by the logarithm of its volume. Indeed, a manifold $M$ with a systole $\sys_1(M)$ contains a ball of radius $r = \sys_1(M)/2$. The volume of a ball in $\Hy^3$ is given by $\vol(B(r)) = \pi(\sinh(2r)-2r)$, hence we get
\begin{align*}
& \vol(M) \ge \pi (\sinh(\sys_1(M)) - \sys_1(M)) \sim \frac{\pi}2 e^{\sys_1(M)};\\
& \log \vol(M) \gtrsim \sys_1(M), \text{ as } \sys_1(M)\to\infty.
\end{align*}
It is well known and was already used in the proof of Proposition~\ref{prop1} that the systole of the congruence covers $M_i\to M$ grows as $i\to\infty$. Together with Theorem~\ref{thm2} this allows us to bound $\Nfr(M_i)$ in terms of $\sys_1(M_i)$, thus proving Gromov's conjecture for the groups under consideration:
\begin{cor}\label{cor:gromov}
For the congruence covers $M_i \to M$ as in Theorem~\ref{thm2}, we have
$$\log \Nfr(M_i) \gtrsim \frac13\,\sys_1(M_i),\text{ as } i\to\infty.$$
\end{cor}

\section{Questions and comments about higher dimensions} \label{sec:questions}
\subsection{} \label{sec:questions1}
We can ask a natural question about generalization of the results of this paper to the higher dimensional spaces (and possibly other word hyperbolic groups). One may conjecture that similar statements to our theorems, propositions, and corollaries should hold in much wider generality. We now briefly review the results which can be used for this purpose and some difficulties arising on the way.

We begin with a general version of Theorem~\ref{thm1}.

\begin{theorem}\label{thm1A}
Let $M$ be a closed $n$-dimensional Riemannian manifold whose curvature satisfies $-K \le K(M) \le - k < 0$. For any $\epsilon > 0$, assuming that the systole $\sys_1(M)$ is sufficiently large, we have
$$\sysg(M) \ge e^{(\frac{\sqrt{k}}{2}-\epsilon)\sys_1(M)}.$$
\end{theorem}

\begin{proof}
As before, let $f: S_g \to M$ be a continuous \pinj map of a surface of genus $g >0$ into $M$. By the classical Eells--Sampson theorem there exists a harmonic map $\bar{f}$ homotopic to $f$. In \cite[Theorem~C.2]{Rez2}, Reznikov showed that $\bar{f}$ satisfies a generalized Thurston's inequality:
$$\area(\bar{f}) \le \frac{4\pi(g-1)}{k}.$$
With this result at hand, we can repeat the rest of the proof of Theorem~\ref{thm1} in the general setting.
\end{proof}

\begin{rmk}
This argument is simpler than the proof in Section~\ref{sec:genus} because it does not require the results of \cite{SY} or \cite{SU} and, of course, we could have used it also for the proof of Theorem~\ref{thm1}. It may occur, however, that the minimal immersions employed there have other interesting properties which may find further applications in the future. This is the main reason why we decided to include a special argument for the $3$-dimensional case. Some of these properties will already come into play in the discussion in Section~\ref{sec:questions2}.
\end{rmk}

The notions of an arithmetic hyperbolic $3$-manifold and its congruence covers are part of the general theory of arithmetic locally symmetric spaces. A short introduction to the theory can be found for example in \cite[Section~10.3]{MR}. Assuming that the reader is familiar with the definitions we will now present a generalization of the results of Section~\ref{sec:congruence}.

\begin{prop} Let $\Gamma$ be a fundamental group of a closed arithmetic hyperbolic $n$-manifold, $n\ge 3$.
\begin{itemize}
\item[(A)] There exists constant $C>0$ such that given a decreasing sequence $\Gamma_i < \Gamma$ of congruence subgroups of $\Gamma$, the corresponding quotient manifolds $M_i = \Hy^n/\Gamma_i$ satisfy
$$\log \sysg(M_i) \gtrsim C \log \vol(M_i), \text{ as } i \to \infty.$$
\item[(B)] Assume that $\Gamma$ is defined by a quadratic form. Then $\Gamma$ has a sequence of congruence subgroups $\Gamma_i$ such that the corresponding quotient manifolds satisfy
$$\sysg(M_i) \lesssim \vol(M_i)^{\frac{6}{n(n+1)}}, \text{ as } i \to \infty.$$
\end{itemize}
\end{prop}

\begin{proof}
(A) This is a generalized version of Proposition~\ref{prop1}. In view of Theorem~\ref{thm1A}, it follows from a lower bound on $\sys_1(M_i)$ in terms of $\log\vol(M_i)$. Such a bound in a general setting can be proved by using the restriction of scalars (see \cite[Section 3.C.6]{Gromov:lectures}). This method, however, does not provide a good estimate for the constant $C$. Computing the explicit constant or proving that it does not depend on the group $\Gamma$ would require a generalization of the main result of \cite{KSV} to the arithmetic hyperbolic $n$-manifolds.

(B) This is a generalization of Proposition~\ref{prop2}. The assumption that $\Gamma$ is defined by a quadratic form implies that $M = \Hy^3/\Gamma$ contains a totally geodesic $2$-dimensional surface $M'$ (see e.g. \cite{Xue}). We can now repeat the argument of Proposition~\ref{prop2} taking into account that if $f$ is a quadratic form of dimension $n+1$, then $|\SO^\circ(f,\cOl/I)| \sim q^{n(n+1)/2}$, where $q$ is the order of $\cOl/I$. So we obtain
$$\area(M'(I)) \sim q^3 \sim \vol(M(I))^{\frac{6}{n(n+1)}}, \text{ as } q \to \infty.$$
\end{proof}

\begin{rmk}
An analogous result can also be stated for the other rank one arithmetic locally symmetric manifolds. The proof of Part~(A) is the same, and Part~(B) differs only in some technical details and the values of the constants.
\end{rmk}

The main challenge for higher dimensions is about extending results of Section~\ref{sec:free}. Recall that here we combined the estimates for the first Betti numbers of congruence covers with a theorem of Baumslag--Shalen in order to show that certain congruence subgroups $\Gamma_i$ are asymptotically $\vol(\Hy^3/\Gamma_i)^{1/3}$-free. Results about the first Betti number are available in much wider generality. The problem is with Theorem~\ref{thm_shalen}, whose proof uses $3$-dimensional hyperbolic geometry in an essential way. It is not known how to generalize this theorem to dimensions higher than $3$.

\subsection{} \label{sec:questions2}
In the proof of Theorem~\ref{thm1} we used Thurston's inequality \eqref{eq2}. There is also an inequality in the other direction which gives a lower bound for the area of any \pinj surface in a hyperbolic $3$-manifold:
\begin{equation}\label{eq5}
2\pi(g-1) \le \area(S_g, \rho).
\end{equation}
A proof of this inequality can be found in \cite[Lemma~6]{Hass}, where it is attributed to an unpublished work of Uhlenbeck. Another proof is given in \cite[Theorem~1]{Rez1}, but it is incomplete and as a result a factor $\frac12$ is missing in the final inequality there. The proof uses stability of minimal surfaces immersed in the manifold and the Gauss--Bonnet theorem. In contrast with the upper bound \eqref{eq2}, this inequality is known only for dimension $n = 3$.

As a corollary of \eqref{eq2} and \eqref{eq5} we can now state an equivalence between $\sysg$ and $\sys_2$ which was mentioned before:
\begin{prop} Let $M$ be a closed hyperbolic $3$-manifold. Then its systolic genus and $2$-dimensional systole satisfy inequalities
\begin{equation*}
2\pi(\sysg(M) - 1) \le \sys_2(M) \le 4\pi(\sysg(M) - 1).
\end{equation*}
\end{prop}
\begin{proof}
Assume that $\sys_2(M)$ is represented by a surface $S_0$ of genus $g_0$ and that $g_1 = \sysg(M)$ is represented by a minimal surface $S_1$. By \cite{KahM}, we have $g_0, g_1 < \infty$, and since $M$ is closed, $g_0, g_1 > 1$. Now inequalities \eqref{eq5} and \eqref{eq2} give
$$
2\pi(g_0 - 1) \le \sys_2(M) \quad\text{and}\quad \area(S_1) \le 4\pi(g_1-1).
$$
Since $g_1 \le g_0$, the first inequality gives $2\pi(g_1-1) \le \sys_2(M)$, while $\sys_2(M) \le \area(S_1)$ together with the second inequality imply that $\sys_2(M) \le 4\pi(g_1 - 1)$.
\end{proof}

A remarkable fact about inequality \eqref{eq5} is that it implies an absolute lower bound for the $2$-dimensional systole  of the closed hyperbolic $3$-manifolds:
\begin{equation} \label{eq_sys}
\sys_2(M) \ge 2\pi.
\end{equation}
This bound demonstrates a sharp contrast with the behavior of $\sys_1(M)$, which can be arbitrarily small, but is in a good agreement with $\sys_3(M) = \vol(M)$, which is bounded below by the Kazhdan--Margulis theorem \cite{KazM}.

In general, we can define $\sys_k(M)$ as the infimum of $k$-volumes of the subsets in $M$ which cannot be homotoped to $(k-1)$-dimensional subsets in $M$ (cf. \cite[Section~3.C.8]{Gromov:lectures}). Our knowledge about the lower bounds for $\sys_k(M)$, $k = 1,2,\ldots, n$ in higher dimensions and for other rank one locally symmetric spaces is very sparse. It was proved only recently that $\sys_1(M)$ of an $n$-dimensional closed hyperbolic manifold $M$ can be arbitrarily small for every $n$ (see \cite{BT} or \cite{BHW}, both developing in different ways an idea suggested by Agol in \cite{Agol} for $n = 4$). From the other hand, by Selberg's conjecture proved by Kazhdan and Margulis in \cite{KazM}, we always have $\sys_n(M)$ bounded uniformly from below.

These observations suggest that given a semisimple (real rank one) Lie group $H$ with the symmetric space $X$, we can define a constant
\begin{align*}
s(H) = \min\{k \mid \sys_k(M) \text{ is uniformly bounded from below for closed manifolds } \\ M\text{ covered by }X\}.
\end{align*}
We will call $s(H)$ by the \emph{Selberg number} of $H$. We expect that if $\sys_k(M)$ is uniformly bounded from below for some $k = k_0$, then it is also uniformly bounded for all $k > k_0$. If this is true, the Selberg number would define a sharp boundary between the two different types of behavior of $k$-dimensional systoles.

From the previous discussion we know that
$$s(\PO(2,1)) = 2,\ s(\PO(3,1)) = 2,\ 1<s(\PO(n,1))\le n, \text{ for } n\ge 4.$$
It would be very interesting to find better bounds for Selberg numbers of higher dimensional hyperbolic spaces. Good numerical values for the lower bounds for $\sys_k(M)$ are also important, but even less is known in this respect.

\end{document}